\documentclass[11pt]{amsart}
\usepackage{amsmath,amssymb,array}
\usepackage{mathtools}
\usepackage[all,cmtip]{xy}

\begin{document}
\DeclarePairedDelimiter{\floor}{\lfloor}{\rfloor}

\newtheorem{thm}{Theorem}[section]
\newtheorem*{thm*}{Theorem}
\newtheorem{lem}[thm]{Lemma}
\newtheorem*{lem*}{Lemma}
\newtheorem{prop}[thm]{Proposition}
\newtheorem*{prop*}{Proposition}
\newtheorem{cor}[thm]{Corollary}
\newtheorem{defn}[thm]{Definition}
\theoremstyle{remark}
\newtheorem*{remark}{Remark}

\numberwithin{equation}{section}

\newcommand{\Z}{{\mathbb Z}} %cph changed from \mathbf
\newcommand{\Q}{{\mathbb Q}}
\newcommand{\R}{{\mathbb R}}
\newcommand{\C}{{\mathbb C}}
\newcommand{\N}{{\mathbb N}}
\newcommand{\FF}{{\mathbb F}}
\newcommand{\fq}{\mathbb{F}_q}
\newcommand{\feq}{\overline{\mathbb F}_q}

\newcommand{\rmk}[1]{\footnote{{\bf Comment:} #1}}

\renewcommand{\mod}{\;\operatorname{mod}}
\newcommand{\ord}{\operatorname{ord}}
\newcommand{\TT}{\mathbb{T}}
\renewcommand{\i}{{\mathrm{i}}}
\renewcommand{\d}{{\mathrm{d}}}
\renewcommand{\^}{\widehat}
\newcommand{\HH}{\mathbb H}
\newcommand{\Vol}{\operatorname{vol}}
\newcommand{\area}{\operatorname{area}}
\newcommand{\tr}{\operatorname{tr}}
\newcommand{\norm}{\mathcal N} % norm =(\frac{ n+\sqrt{n^2-4}} 2)^2
\newcommand{\intinf}{\int_{-\infty}^\infty}
\newcommand{\ave}[1]{\left\langle#1\right\rangle} %  average
\newcommand{\Var}{\operatorname{Var}}
\newcommand{\Prob}{\operatorname{Prob}}
\newcommand{\sym}{\operatorname{Sym}}
\newcommand{\disc}{\operatorname{disc}}
\newcommand{\delt}{\operatorname{\delta}}
\newcommand{\tdeg}{\operatorname{tot.deg}}
\newcommand{\pisc}{\operatorname{disc}_{+}}
\newcommand{\Berl}{\operatorname{Berl}}
\newcommand{\hgt}{\operatorname{ht}}
\newcommand{\gal}{\operatorname{Gal}}
\newcommand{\CA}{{\mathcal C}_A}
\newcommand{\cond}{\operatorname{cond}} % conductor
\newcommand{\lcm}{\operatorname{lcm}}
\newcommand{\Kl}{\operatorname{Kl}} %Kloosterman sum
\newcommand{\leg}[2]{\left( \frac{#1}{#2} \right)}  % Legendre symbol

\newcommand{\sumstar}{\sideset \and^{*} \to \sum}

\newcommand{\LL}{\mathcal L} %L-function of u
\newcommand{\sumf}{\sum^\flat}
\newcommand{\Hgev}{\mathcal H_{2g+2,q}}
\newcommand{\USp}{\operatorname{USp}}
\newcommand{\conv}{*}
\newcommand{\dist} {\operatorname{dist}}
\newcommand{\CF}{c_0} % Fejer constant
\newcommand{\kerp}{\mathcal K}

\newcommand{\fs}{\mathfrak S}
\newcommand{\rest}{\operatorname{Res}} % resultant
\newcommand{\af}{\mathbb A} % affine line
\let\emptyset\varnothing
\title[Chowla's conjecture for the rational function field 2]
{The autocorrelation of the M\"obius function and Chowla's
conjecture for the rational function field in characteristic 2}
\author{Dan Carmon}
\address{Raymond and Beverly Sackler School of Mathematical Sciences,
Tel Aviv University, Tel Aviv 69978, Israel}

\date{\today}

\thanks{The research leading to these results has received funding from the European
Research Council under the European Union's Seventh Framework Programme
(FP7/2007-2013) / ERC grant agreement n$^{\text{o}}$ 320755.}
%\thanks{This work was   supported by  the Israel Science Foundation (grant No.
%1083/10).}
\begin{abstract}
We prove a function field version of Chowla's conjecture on the
autocorrelation of the M\"obius function in the limit of a large
finite field of characteristic 2.
\end{abstract}
\maketitle

\section{Introduction}

Let $\fq$ be a finite field  of $q$ elements, and let $\fq[x]$ be the
polynomial ring over $\fq$. The M\"obius function of a non-zero
polynomial $F\in \fq[x]$ is defined to be $\mu(F)=(-1)^r$ if
$F=cP_1\dots P_r$ with $0\neq c\in \fq$ and $P_1,\dots, P_r$ are
distinct monic irreducible polynomials, and $\mu(F)=0$ otherwise.
Let $M_n\subset \fq[x]$ be the set of monic polynomials of degree
$n$ over $\fq$, which is of size $\#M_n = q^n$.

For $r>0$,  distinct polynomials  $\alpha_1,\dots, \alpha_r \in
\fq[x]$ with $\deg \alpha_j<n$, and
 $\epsilon_i\in \{1,2\}$, not all even, set
\begin{equation}
  C(\alpha_1,\dots, \alpha_r;n):=
\sum_{F\in M_n} \mu(F+\alpha_1)^{\epsilon_1}\dots
\mu(F+\alpha_r)^{\epsilon_r}
\end{equation}

In \cite{Carmon Rudnick}, an upper bound on $|C(\alpha_1,\dots, \alpha_r;n)|$ was established for fields 
of odd characteristic, demonstrating that for any fixed $n>1$, $r>1$, we have  
$\lim_{q\to \infty} \frac 1{\#M_n}\sum_{F\in M_n}
C(\alpha_1,\dots, \alpha_r;n) = 0$. This is analogous to Chowla's conjecture over function fields,
in the limit of a large finite field.

This result has since found further applications. In \cite{LBS}, Bary-Soroker
utilizes a result similar to a part of the proof, named {\em square independence}, and computes a certain Galois group to be $S_n^r$. 
This computation then implies many equidistribution and independence results, proving function field analogues,
in the limit of a large base field, 
to myriad classical problems, such as the Hardy-Littlewood conjecture, and the additive and Titchmarsh divisor problems.
See \cite{ABSR} for more details and examples. 
%We note that the analogue to Chowla's conjecture could also be seen as a consequence of this computation, 
%although the implied asymptotics would be worse than those obtained directly in \cite{Carmon Rudnick}.
We stress that Bary-Soroker's computation, and any implications thereof, were only valid in odd characteristic,
due to square independence having been established only in odd characteristic.

In this paper, we shall provide a bound on $|C(\alpha_1,\dots, \alpha_r;n)|$ in the case of characteristic 2,
yielding the analogue to Chowla's conjecture in this setting. 
We shall also verify square independence in characteristic 2, thus extending the validity of Bary-Soroker's computation, and all
its implications.

Henceforth, we shall assume that $q$ is even. 
As in odd characteristic, for $r=1$ and $n>1$, we have $\sum_{F\in M_n} \mu(F) = 0$.
For $n=1$ we have $\mu(F)\equiv -1$ and the sum equals $(-1)^{\sum
  \epsilon_j} q$.
The case $n=2$ is a new special case in characteristic 2, and will be handled separately, in section \ref{sec:n2}.
For $n>2$, $r>1$ we show:

\begin{thm}\label{main thm}
Fix $r>1$ and assume that $n>2$ and that $q$ is even.
%Then for fixed $n$,
Then for any choice of distinct polynomials  $\alpha_1,\dots,
\alpha_r \in \fq[x]$ with $\max \deg \alpha_j<n$,  and
$\epsilon_i\in \{1,2\}$, not all even,
\begin{equation}\label{main thm eq}
| C(\alpha_1,\dots, \alpha_r;n)| \leq rn q^{n-\frac 12} + \frac 34 (r+3)n^2
q^{n-1}
\end{equation}
\end{thm}

\section{Analogues in Characteristic 2}
The starting point in \cite{Carmon Rudnick} was Pellet's formula, expressing
the M\"obius function in terms of the quadratic character of the discriminant:
\begin{equation*}
 \mu(F) = (-1)^{\deg F}\chi_2(\disc F)
\end{equation*}
For even $q$, Pellet's formula does not hold; indeed, even the usual quadratic
character $\chi_2$ itself is meaningless, as every element of $\fq$ is the square
of another. There is, however, a similar formula, which utilises Berlekamp's
discriminant (first defined in \cite{Berlekamp}). We shall repeat here the definitions and required properties of Berlekamp's discriminant.

\subsection{Definition of Berlekamp's discriminant}
Given a polynomial \linebreak
$F(x)=a_nx^n+a_{n-1}x^{n-1}+\dots+a_0$, $a_n\neq 0$ with coefficients
in $\fq$, let $r_1,\dots, r_n$ be its roots in some algebraic extension of $\fq$. 
The Berlekamp discriminant of $F$ is defined in terms of its roots as
\begin{equation}\label{Berl formula}
\Berl(F) = \sum_{i<j} \frac{r_ir_j}{r_i^2+r_j^2}
\end{equation}
The expression $\Berl(F)$ is symmetric in the roots of $F$, hence it is in $\fq$
and its value is independent of the extension used. Furthermore, taking 
a common denominator, we may write
\begin{equation}\label{Berl formula ex}
\Berl(F) = \frac{a_n^{2n-2} \displaystyle\sum_{i<j} \big(r_ir_j \displaystyle\prod_
  {\substack{i'<j'\\(i',j')\neq(i,j)}} (r_{i'}^2+r_{j'}^2)\big)}{a_n^{2n-2} \displaystyle\prod_{i<j} (r_i^2+r_j^2)}
\end{equation}
Note that both the denominator and the numerator are symmetric polynomials in the roots of $F$. 
Hence, they are homogeneous polynomials\footnote{
It is perhaps not trivial that they are indeed polynomials, rather than rational functions with a power of $a_n$ in their denominators. 
We will see they are indeed polynomials as a by-product of their computation. Furthermore, in all of our applications, $F$ will be monic.
} (over $\mathbb{F}_2$) in the coefficients of $F$
of degree $2n-2$, and furthermore the denominator is in fact the discriminant of $F$, for in characteristic 2,
$r_i^2+r_j^2 = (r_i-r_j)^2$. Following Berlekamp, we denote the numerator of $\Berl(F)$ by $\xi(F)$, that is,
\begin{equation}
\Berl(F) = \frac{\xi(F)}{\disc F}
\end{equation}
Note also that in characteristic 2, $\disc F = \delt(F)^2$ where 
\begin{equation*}
\delt(F) =  a_n^{n-1}\prod_{i<j}(r_i+r_j) 
\end{equation*}
is a polynomial in the coefficients of $F$
with total degree $n-1$, and degree at most $d(n) = \floor{\frac{n-1}{2}}$ in $a_0$ -- its leading term,
as a polynomial in $a_0$, is $a_n^{\frac{n-1}{2}}a_0^{\frac{n-1}{2}}$ for odd $n$, and 
$a_{n-1}^{\frac{n}{2}}a_0^{\frac{n-2}{2}}$ for even $n$.
The formulae for Berlekamp's discriminant for degrees up to 3 are:
\begin{align*}
 \Berl(ax+b)&= \tfrac01 \\
 \Berl(ax^2+bx+c)&= \frac{ac}{b^2} \\
 \Berl(ax^3+bx^2+cx+d)&= \frac{a^2d^2+abcd+b^3d+ac^3}{(ad+bc)^2}
\end{align*}

\subsection{Effective computation of Berlekamp's discriminant}
The formulae above do not lend themselves immediately to computations of $\Berl(F)$ or $\xi(F)$ 
in terms of the coefficients of $F$.
A computational method can be obtained by first lifting the coefficients of $F$ from $\fq$ to a
field with characteristic 0. To do so, choose an algebraic extension $K$ of the rationals
that becomes isomorphic to $\fq$ when reduced modulo 2, and choose any lifting 
$F_0(x)=a_{0,n}x^n+a_{0,n-1}x^{n-1}+\dots+a_{0,0}$ with coefficients in $K$ such that
$F_0 \equiv F \pmod 2$. If the roots of $F_0$ in
the algebraic closure are $r_{0,1},\dots, r_{0,n}$ 
(such that $r_{0,i} \equiv r_i \pmod 2)$ we have
$\disc F_0 = a_{0,n}^{2n-2}\prod_{i<j}(r_{0,i}-r_{0,j})^2$. We define similarly
$\pisc F_0 = a_{0,n}^{2n-2}\prod_{i<j}(r_{0,i}+r_{0,j})^2$. Note that $\pisc F_0$ is a
symmetric polynomial in the roots of $F_0$, and thus is an integral
polynomial in the coefficients of $F_0$ -- in fact, it is the square of
such a polynomial. Furthermore, it is clear that
$\disc F_0 \equiv \pisc F_0 \pmod 4$, as polynomials. Therefore the expression 
$\xi_0(F_0) = \frac{\pisc F_0 - \disc F_0}{4}$ is also an integral polynomial 
in the coefficients of $F_0$. It is now easy to verify that when $\xi_0$ is reduced modulo 2,
the obtained polynomial must indeed be equal to $\xi(F)$, as given by the numerator of 
formula \eqref{Berl formula ex}. In particular, the result of this process is independent of the lifting.

In fact, in our computations we will use a simpler lifting.
We shall always compute the discriminants in general settings, where all of the coefficients
of $F$ are either 0, 1 or a symbol from a set of variables $V$, never any explicit value in $\fq$. Thus
we need only to lift the coefficients from $\FF_2[V]$ to $\Z[V]$, which 
can be done, for example, by lifting $0,1 \in \FF_2$ to $0,1 \in \Z$. 
The discriminants computed in this manner will surely yield the correct polynomials
in $\FF_2[V]$, and hence the same formulae will also be valid for any substitution of values 
from $\fq$ into the variables of $V$.

Finally, we note that $\pisc(F_0)$, like the discriminant, can be expressed as a resultant,
and is therefore easily computable:
\begin{equation}\label{pisc formula 1}
 \pisc(f(x)) = \frac{\rest(f(x),f(-x))}{2^n a_0 a_n} = \rest\left(\frac{f(x)-f(-x)}{2x},\frac{f(x)+f(-x)}{2}\right)
\end{equation}
Splitting $f$ into its even and odd part as $f(x) =  g(x^2)+xh(x^2)$, we may rewrite formula
 \eqref{pisc formula 1} as
\begin{equation}\label{pisc formula 2}
 \pisc(f(x)) = \rest(h(x^2),g(x^2)) = \rest(h(x),g(x))^2
\end{equation}
Using these methods, we show that for any $n>2$, the degree of
$\xi(F)$ in $a_0$ is at most $2d(n)$; this will be important later on. For all $n$, both $\disc(F_0),\pisc(F_0)$ have degree at most
$n-1$ in $a_0$. For odd $n$, $2d(n) = n-1$, and we are done. For even $n$, $2d(n) = n-2$, and so we must check
that the coefficients of $a_0^{n-1}$ in $\xi(F)$ vanish. And indeed, $\pisc(F_0)$ is a square, hence its 
degree in $a_0$ must be even, and therefore less than $n-1$. On the other hand, the leading coefficient of $\disc(F_0)$
is known to be $\pm n^na_n^{n-1}a_0^{n-1}$, which is clearly $0 \mod 8$ for any $n > 2$. Note that this is 
false for $n = 2$, and indeed we find that $\xi(F) = a_2a_0$ while $2d(2) = 0$.

\subsection{An analogue to Pellet's formula}
The main theorem in \cite{Berlekamp} provides an analogue to Pellet's formula 
in characteristic 2. We restate it here in more familiar terms.
Let $\chi_2 : \fq \to \{\pm 1\}$ be defined by $\chi_2(x) = 1$ iff
 $x = y^2 + y$ for some $y \in \fq$, and $\chi_2(x) = -1$ otherwise.
Note that the map $y \mapsto y^2 + y$ is linear over $\FF_2$,
and its kernel is the set $\{0,1\}$. Therefore its image is a $\FF_2$-linear subspace
of $\fq$ with codimension 1, and thus $\chi_2$ is a group homomorphism,
i.e. $\chi_2(x+y) = \chi_2(x)\chi_2(y)$. We are
interested in evaluating $\chi_2(\Berl F)$. From \eqref{Berl formula}
we may write $\Berl F  = \beta^2 + \beta$, where
 $\beta = \sum_{i<j} \frac{r_i}{r_i+r_j} \in \FF_{q^2}$, so we need only determine whether
$\beta \in \fq$. Note that any odd permutation of the roots $r_i$ changes
 $\beta$ to $\beta + 1$, and that $\beta$ is fixed under any even permutation.
As $\beta \in \fq$ iff $\beta$ is fixed under the Frobenius endomorphism,
the value of $\chi_2$ is determined by the sign of the permutation on $r_i$
given by the Frobenius endomorphism. The following analogue to Pellet's
 formula is now immediate:
\begin{equation}\label{Pellet2}
\mu(F) = (-1)^{\deg F} \chi_2(\Berl(F))
\end{equation}
Note that the formula is only true whenever $F$ is squarefree. Indeed,
otherwise $\disc F = 0$, in which case $\Berl(F)$ is not even properly defined.
Somewhat informally, we may correct this by assigning $\chi_2(\infty) = 0$. However, 
this difficulty is more easily avoided by assuming that $\disc (F+\alpha_i) \neq 0$ for all
$i$. Indeed, as there are exactly $q^{n-1}$ polynomials in $M_n$ with $\disc F = 0$,
this assumption fails for at most $rq^{n-1}$ $n$-tuples $a_0,\dots,a_{n-1}$. This negligible error will be
collected into the $\frac34(r+3)n^2q^{n-1}$ term further along.
We may henceforth assume that all $\epsilon_i = 1$, as terms with $\epsilon_i = 2$ do not affect the remaining
non-vanishing summands at all.

\section{Reduction to a counting problem}
Continuing analogously to \cite{Carmon Rudnick}, we may write
\begin{equation}
C(\alpha_1,\dots, \alpha_r;n) = (-1)^{nr} \sum_{F\in M_n}
\chi_2\big(\!\Berl(F+\alpha_1)+\dots
+\Berl(F+\alpha_r) \big)
\end{equation}
We single
out the constant term $t:= F(0)$ of $F\in M_n$ and write $F(x) = f
(x)+t$, with
\begin{equation}
f(x) = x^n+a_{n-1} x^{n-1}+\dots +a_1 x
\end{equation}
and set
\begin{equation}
B_f(t) := \Berl(f(x)+t) =  \frac{\xi(f(x)+t)}{\disc(f(x)+t)} = \frac{\xi_f(t)}{D_f(t)} = \frac{\xi_f(t)}{\delta_f^2(t)}
\end{equation}
which is a rational function of height\footnote{
The height $\hgt(p)$ of a rational function is the maximum of the degrees
of its numerator and denominator. It is equal to the total order of its poles (resp. zeros),
including poles (resp. zeros) at infinity.} at most $n-1$ in $t$. Therefore we have
\begin{equation}\label{charsum}
\left| C(\alpha_1,\dots, \alpha_r;n) \right|  \leq \sum_{a\in \fq^{n-1}}
\left|\sum_{t\in \fq} \chi_2\big( B_{f+\alpha_1}(t)+\dots
+ B_{f+\alpha_r}(t)\big) \right|
\end{equation}

In order to bound the character sum, we apply Weil's theorem to the appropriate Artin-Schreier curve. 
See \cite[Theorem 1]{Perelmuter} for the general claim and proof; we state it here for characteristic 2.
\begin{thm}\label{weil 2}
 Let $\fq$ be a field of characteristic 2, and let $p \in \fq(t)$ 
 be a rational function which is not of the form $H^2(t) + H(t) + c$ for any
 $H(t) \in \fq(t)$, $c \in \fq$.
 Starting with the projective curve $y^2 + y = p(t)$, using translations 
 of the variable $y$ by appropriate rational functions in $\fq(t)$,
 we may obtain an isomorphic curve $y^2 + y = \tilde p(t)$, satisfying:
 \begin{enumerate}
  \item $\tilde p(t) = p(t) + Q^2(t) + Q(t)$ for some rational function $Q(t) \in \fq(t)$ with $\hgt(Q) \le \tfrac12 \hgt(p)$
  \item The poles $P_0,\dots,P_{s}$ of $\tilde p$ are all poles of $p$
  \item The order $d_i$ of the pole $P_i$ in $\tilde p$ is less than or equal to its order in $p$
  \item The orders $d_i$ are all odd.
 \end{enumerate}
 The following bound then holds:
 \begin{equation}\label{weil-as}
  \left| \sum_{t\in \fq}\chi_2(\tilde p(t)) \right| \leq 2gq^{1/2} + 1
 \end{equation}
 Where g is the genus of the (isomorphic) curves, given by
 \begin{equation}\label{genus formula}
  g = \frac{\sum_{i=0}^s (d_i+1) - 2}{2}
 \end{equation}
\end{thm}
Note that the condition $p \neq H^2(t) + H(t) + c$ was necessary (and sufficient) in
order to ensure that $\tilde p$ is not a constant function. 

In our case, we want a bound for the character sum of $p$, not of $\tilde p$.
Note that $\chi_2(Q^2(t)+Q(t)) = 1$ whenever $t$ is not a pole of $Q$, and equals $0$ at poles.
Therefore $\chi_2(\tilde p(t))$ and $\chi_2(p(t))$ may differ only at the poles of $Q$, 
and if they do, they differ by at most 1. Thus
\begin{equation}\label{pptilde}
 \left| \sum_{t\in \fq}\chi_2(p(t)) \right| \le \left| \sum_{t\in \fq}\chi_2(\tilde p(t)) \right| + \hgt(Q(t))
\end{equation}

We will use the following easy corollary of \eqref{genus formula} to estimate the genus:
\begin{cor}\label{genus bound}
 Let $\fq$ be a field of characteristic 2, and let $p \in \fq(t)$
 be a rational function not of the form $H^2(t) + H(t) + c$. 
 Suppose that the order of $p$ in all of its poles is even, except in at most one pole.
 Then $g \le \frac{\hgt(p)-1}{2}$, where $g$ is the genus of the curve $y^2 + y = p(t)$.
\end{cor}
For us, the relevant rational function is $p(t) = B_{f+\alpha_1}(t)+\dots
+B_{f+\alpha_r}(t)$, which has height at most $r(n-1)$.
The denominator of $p$ is $(\delta_{f+\alpha_1}\dots\delta_{f+\alpha_r})^2$,
which indeed shows that all its poles have even order, 
with the sole possible exception of a pole at infinity.
Hence $p$ satisfies the conditions of Corollary \ref{genus bound}, 
provided $p(t) \neq H^2(t) + H(t) + c$.
Combining \eqref{weil-as}, \eqref{pptilde} and Corollary \ref{genus bound}, we obtain
\begin{equation}
\left| \sum_{t\in \fq}\chi_2(p(t)) \right| \le 2gq^{1/2} + 1 + \hgt(Q) < \hgt(p)q^{1/2} + \tfrac{\hgt(p)}2 + 1 < rnq^{1/2}
\end{equation}
which, when applied to \eqref{charsum}, provides the major term in \eqref{main thm eq}. 

We need now only find a way to bound the size of the set $G_n^c$ of ``bad'' $a$'s where 
$p(t) = H^2(t) + H(t) + c$. As in odd characteristic, we cover $G_n^c$ by
simpler, algebraic varieties. 

\subsection{Covering $G_n^c$}
Suppose WLOG that $\alpha_1 = 0$.
\begin{prop}\label{main prop}
We can write $G^c_n\subset A_n\cup B_n\cup C_n$ where:
\begin{itemize}
\item
$A_n$ is the set of those $a\in \fq^{n-1}$ for which $\deg \delta_f = d(n)$ and $\delta_f$ is not coprime to $\xi_f \delta'^2_f - \xi'^2_f$, that is
\begin{equation}
A_{n}  = \{a\in \fq^{n-1}: \deg \delta_f = d(n), \rest(\delta_f,\xi_f \delta'^2_f - \xi'^2_f) =0 \}
\end{equation}

\item
$B_n=\cup_{j\neq 1} B(j)$ where $B(j)$ is the set of those $a\in \fq^{n-1}$ for which $\deg \delta_f = \deg \delta_{f+\alpha_j} = d(n)$ and
  $\delta_f(t)$ and $\delta_{f+\alpha_j}(t)$ have a common zero, that is
\begin{equation}
B(j) = \{a\in \fq^{n-1}: \deg \delta_f = \deg \delta_{f+\alpha_j} = d(n), \rest(\delta_f(t), \delta_{f+\alpha_j}(t)
) = 0 \}
\end{equation}
\item
$C_n = \cup_{j} C(j)$ where %$C(j)$ is the set of those $a\in \fq^{n-1}$ for which 
%$\deg \delta_{f+\alpha_j} < d(n)$, that is
\begin{equation}
C(j)  = \{a\in \fq^{n-1}: \deg \delta_{f+\alpha_j} < d(n)  \}
\end{equation}
\end{itemize}
\end{prop}
Henceforth, let us denote $\Xi_f = \xi_f \delta'^2_f - \xi'^2_f$.
\begin{proof}

We will assume $a \in G^c_n \setminus (B_n \cup C_n)$, and show that $a \in A_n$. By $a \not\in B(j)\cup C(j)\cup C(1)$,
$\delta_f$ is coprime to $\delta_{f+\alpha_j}$ for all $j\neq 1$. From $a \in G^c_n$ we obtain
\begin{equation}\label{bad polynomial}p(t) = 
\frac{\xi_{f+\alpha_1}(t)}{D_{f+\alpha_1}(t)} + \dots + \frac{\xi_{f+\alpha_r}(t)}{D_{f+\alpha_r}(t)} =
H^2(t) + H(t) + c.
\end{equation}
Note that from $a \not\in C(1)$ we have $\deg D_f = 2d(n) \ge \deg \xi_f$.
Consider a root of $\delta_{f}$ with multiplicity $m$. Then it is a pole of $p$ with multiplicity at most $2m$. Hence
it is a pole of $H$ with multiplicity at most $m$, its multiplicity in $\delta_f$.
As this is true for every root of $\delta_{f}$, it follows that we may write
$H = \frac{H_1}{\delta_{f}H_d}$ where $H_1, H_d$ are polynomials and $H_d$ is coprime to $\delta_f$.
Thus there exists a unique polynomial $H_D$ with $\deg H_D < \deg \delta_f$ such that
$H_D \equiv \frac{H_1}{H_d} \pmod {\delta_f}$. We may then write $H_1 = H_DH_d+H_2\delta_f$ for 
some polynomial $H_2$, or equivalently, $H = \frac{H_D}{\delta_f} + \frac{H_2}{H_d}$. Substituting this
relation in \eqref{bad polynomial}, we obtain from $\deg H_D < \deg \delta_f$ and $\deg \xi_f \leq \deg D_f$,
as well as $\delta_f$ being coprime to $D_{f+\alpha_j}$ and $H_d$, that 
\begin{equation}
\frac{\xi_f(t)}{\delta_f^2(t)} = \left(\!\frac{H_D}{\delta_f}\!\right)^2 + \frac{H_D}{\delta_f} + c_2.
\end{equation}
Multiplying by $\delta_f^2$, we obtain
\begin{equation}\label{HD eq 1}\xi_f = H_D^2 + \delta_f H_D + c_2\delta_f^2\end{equation}
Differentiating the last formula, we get
\begin{equation}\label{HD eq 2}\xi'_f = \delta'_f H_D + \delta_f H'_D\end{equation}
Reducing equations \eqref{HD eq 1} and \eqref{HD eq 2} modulo $\delta_f$, we find
\begin{align}
 \xi_f \equiv H_D^2 \pmod {\delta_f} \\
 \xi'_f \equiv \delta'_fH_D \pmod {\delta_f}
\end{align}
From which we easily derive 
\begin{equation}\label{modular condition}
 \xi'^2_f \equiv \delta'^2_fH^2_D \equiv \delta'^2_f\xi_f \pmod {\delta_f}
\end{equation}
Congruence \eqref{modular condition} states that
$\delta_f$ must divide $\Xi_f = \xi_f \delta'^2_f - \xi'^2_f$.
 $a \not \in C(1)$ implies in particular that $\delta_f$ is not constant, and therefore $\delta_f$ and $\Xi_f$ are not coprime -- 
thus $a \in A_n$ by definition. 
\end{proof}
\subsection{Bounding degrees and sizes}
In order to complete the proof we need to provide bounds for the
degrees of the polynomials defining $A_n$ and $B_n$, and to show that these
polynomials are not identically zero. We must also bound the size of $C_n$.
We shall first obtain the bound on the degrees and the sizes of $A_n$, $B_n$,
assuming the relevant polynomials do not vanish.
We begin with the following lemma.
\begin{lem}\label{deg lem}
 Let $A = A' \cup \{t\}$ be a set of variables. Let $f,g \in \fq[A]$ be homogeneous polynomials in 
 the variables $A$ of degrees $d_f, d_g$ respectively. Let $n_f, n_g$ be their respective degrees as
 polynomials in the variable $t$ with coefficients in $\fq[A']$. Set $R = \rest_t(f,g) \in \fq[A']$.
 Then $R$ is a homogeneous polynomial in the variables $A'$ of degree $ d_fn_g + n_fd_g - n_fn_g $. 
\end{lem}
\begin{proof}
 Write $f = \sum_{k=0}^{n_f} a_k t^k, g = \sum_{k=0}^{n_g} b_k t^k$, where $a_k, b_k$ are homogeneous polynomials
 in the variables $A'$ of degrees $d_f - k, d_g - k$, respectively. Let $\prod a_k^{r_k} \prod b_k^{s_k}$ be any
 arbitrary monomial appearing in $R$. By well known properties of the resultant, we have
 %$$\sum_{k=0}^{n_f} {r_k} = n_g,\quad \sum_{k=0}^{n_g} {s_k} = n_f, \quad
 %\sum_{k=0}^{n_f} {k\ r_k} +\sum_{k=0}^{n_g} {k\ s_k} = n_f n_g $$
 $$\sum {r_k} = n_g,\quad \sum {s_k} = n_f, \quad
 \sum {k(r_k+s_k)} = n_f n_g. $$
 It follows that the total degree of the monomial in the variables $A'$ is
%  \begin{align*}
%   & \sum_{k=0}^{n_f} {(d_f - k)r_k} +\sum_{k=0}^{n_g} {(d_g-k)s_k} =  \\
%   & d_f \sum_{k=0}^{n_f} {r_k} + d_g \sum_{k=0}^{n_g} {s_k} - 
%   \left(\sum_{k=0}^{n_f} {k\ r_k} +\sum_{k=0}^{n_g} {k\ s_k}\right) =  
%   n_fd_g + d_gn_f - n_fn_g 
%  \end{align*}
 \begin{align*}
  & \sum{(d_f - k)r_k} +\sum {(d_g-k)s_k} =  \\
  & d_f \sum {r_k} + d_g \sum {s_k} - \sum {k(r_k+s_k)} =  
  d_fn_g + n_fd_g - n_fn_g 
 \end{align*}
 as claimed.
\end{proof}

In order to obtain bounds on the sizes of $A_n, B_n$ from the degrees of their defining polynomials,
we will use of the following elementary lemma (\cite[\S 4, Lemma 3.1]{Schmidt}):
\begin{lem}
Let $h(X_1,\dots, X_m) \in \fq[X_1,\dots, X_m]$ be a non-zero
polynomial of total degree at most $d$. Then the number of zeros  of
$h(X_1,\dots, X_m)$ in $\fq^m$ is at most
\begin{equation}
\#\{x\in \fq^m: h(x)=0\} \leq d q^{m-1}.
\end{equation}
\end{lem}

% Note that, by their definitions, $n_f \le d_f, n_g \le d_g$. It is then clear that 
% $d_R = d_fn_g + n_fd_g - n_fn_g = d_fd_g - (d_f-n_f)(d_g-n_g)$
% does not decrease when any one of $n_f,n_g,d_f,d_g$ is increased, as long as $n_f \le d_f, n_g \le d_g$
% is maintained. We thus obtain the following corollary: 
% \begin{cor}\label{deg cor}
%  Let $A,A',t,f, g, R$ be defined as in Lemma \ref{deg lem}, but suppose now that $f,g$ are homogeneous
%  in $A$ of degrees {\em at most} $d_f, d_g$, and similarly their degrees in $t$ are at most
%  $n_f, n_g$. Suppose also that $n_f \le d_f, n_g \le d_g$. Then $R$ is homogeneous
%  in $A$ of degree at most $n_fd_g + d_gn_f - n_fn_g$. \qed
% \end{cor}

\subsection{Bounding $C_n$}
We note that for odd $n > 2$, $\delta_{f+\alpha_j}(t)$ is always of degree exactly $(n-1)/2$.
For even $n > 2$, $\deg \delta_{f+\alpha_j}(t) = (n-2)/2$
iff the coefficient of $x^{n-1}$ in $f+\alpha_j$ is non-zero. This is true simultaneously for every $j$
for all but at most $rq^{n-2}$ tuples $a$ where $a_{n-1} \in \{\alpha_{j,n-1}\}$.  Hence $\#C_n < rq^{n-2}$. This contribution will be merged into the $B_n$ bound.
\subsection{Bounding $B_n$}
$\delta_f(t), \delt_{f+\alpha_j}(t)$ have total degree $n-1$ in $a_{n-1},\dots,a_1,t$, and by definition of $B(j)$,
they have degree $d(n)$ as polynomials in $t$. Hence by Lemma 
\ref{deg lem}, $\rest(\delta_f(t), \delta_{f+\alpha_j}(t))$ has total degree $2(n-1)d(n) - d(n)^2$
in the coefficients $a_{n-1},\dots,a_1$, which equals $\frac{3}{4}(n-1)^2$ for odd $n$
and $\frac{(3n-2)(n-2)}{4}$ for even $n$; in either case, we may round this up to $\frac34n^2$ and obtain\footnote{
Note that, as $n \ge 3$, the process of rounding up to $\frac34n^2$ adds at least $\frac{15}4 > 2$. 
This, together with the rounding of the bound on $\#A_n$, covers the two instances where we neglected an error 
of $rq^{n-1}$: The first in the assumption that $\delta(F+\alpha_j) \neq 0$, the second in bounding $C_n$.}
\begin{equation}\label{B_n bound}
\#B_n < \frac{3}{4}(r-1)n^2q^{n-2}
\end{equation}
\subsection{Bounding $A_n$} We have seen that the degree of $\xi_f$
in $t$ is at most $2d(n)$ and that the total degree of $\xi_f$ in
$a_n,a_{n-1},\dots,a_1,t$ is $2(n-1)$.
$\delta_f$ has total degree $n-1$ and degree $d(n)$ in $t$.
Thus we find that the degree of $\Xi_f$ in $t$ is at most $4d(n)-2$
(i.e. $2n-4$ for odd $n$ and $2n-6$ for even $n$), and its total degree in $a_n,a_{n-1},\dots,a_1,t$
is exactly $4n-6$. As the degree of $\delta_f$ is constant, we may assume that 
the polynomial given by $\rest(\delta_f,\Xi_f)$ is fixed 
by always assuming $\Xi_f$ is of degree exactly $4d(n)-2$.
This is valid, as adding leading zeros to only one of the polynomials multiplies the resultant by a 
non-zero factor. We now have
\begin{align*}
  & \deg_t \delt_f = d(n), &
  & \tdeg \delt_f = n-1 \\
  & \deg_t \Xi_f = 4d(n)-2, &
  & \tdeg \Xi_f = 4n-6
\end{align*}
Hence by Lemma \ref{deg lem}, the degree of $\rest(\delta_f,\Xi_f)$ in 
$a_{n-1},\dots,a_1$ is \linebreak
$(n-1)(4d(n)-2)+d(n)(4n-6)-d(n)(4d(n)-2)$,
which is equal to $(n-1)(3n-5)$ for odd $n$ and $3n^2-10n+6$ for even $n$. In either case,
we may round this up to $3n^2$ and obtain 
\begin{equation}\label{A_n bound}
\#A_n < 3n^2q^{n-2}
\end{equation}
and combining this with \eqref{B_n bound} we get
\begin{equation}
\#G_n^c < \frac{3}{4}(r+3)n^2q^{n-2}
\end{equation}
proving theorem \eqref{main thm}.
\section{Non-vanishing of the resultants}

\subsection{Non-vanishing of the polynomials defining $B_n$}
\begin{prop}\label{prop:resultant}
Given a non-zero polynomial $\alpha\in \fq[x]$ with $\deg \alpha<n$,
the function $a\mapsto \rest(\delta_f(t),\delta_{f+\alpha}(t))$ is not the zero
polynomial,
that is, the polynomial function
\begin{equation}
  R(a):=\rest_t(\delta_f(t),\delta_{f+\alpha}(t))\in \FF_2[\vec a]
\end{equation}
is not identically zero.
\end{prop}
We note that the proof of the analogous proposition in odd characteristic \cite[Proposition 3.1]{Carmon Rudnick} did not
in fact rely on the characteristic being odd. We may follow the same arguments
to see again that $R(a)$ cannot be identically zero. More accurately, the proof in \cite{Carmon Rudnick} referred to
the polynomial $\rest_t(D_f(t),D_{f+\alpha}(t))$ which, in characteristic 2, is equivalent to
$\rest_t(\delta_f^2(t),\delta_{f+\alpha}^2(t)) = R(a)^4$. The main observation behind the proof
was that the roots of $D_f(t)$, which are the same as the roots of $\delta_f(t)$, are exactly those $t$
for which there exists some $\rho$ 
(in some fixed algebraic closure of $\fq$) that satisfies 
$f'(\rho) = 0$ and $t=-f(\rho)$. This observation is just as valid in characteristic 2,
as are the calculations that followed.
This completes the proof of inequality \eqref{B_n bound}.
\subsection{Non-vanishing of the polynomials defining $A_n$}
We wish to show that the algebraic condition for being in $A_n$, i.e.
$\rest(\delta_f,\Xi_f) = 0$, is not always satisfied.
We will demonstrate this by giving explicit examples of $f$ that do not satisfy the equation.\footnote{
A different approach was used in the case of odd characteristic. 
The approach used here could have been applied there partially. For example, the polynomial
$f+t = x^n+ax+t$ yields $D_f(t) = (-1)^{\frac{(n-1)n}{2}}(n^nt^{n-1}+(1-n)^{n-1}a^n)$, which satisfies
$\disc_t(D_f(t)) \neq 0$ given $\gcd(q,n(n-1)) = 1, a \neq  0$. While this covers many cases, 
the remaining cases are not as easily dealt with. 
The algebraic approach managed to avoid this division into cases completely.}
We will construct two generic examples, depending on the parity of $n$.
The relevant computations are given in more detail in section \ref{sec:computation}.

Consider first $n \ge 3$ odd, and take $f+t = x^n + ax^2 + t$. An easy computation
 then yields
\begin{equation*}
\begin{split}
 \delta_f & = t^{\frac{n-1}{2}} \\
 \xi_f & = \frac{(-1)^\frac{n-1}{2}n^n-1}{4}t^{n-1}+a^nt \\
 \xi'_f &= a^n
\end{split}
\end{equation*}

Note that the only root of $\delta_f$ is at $t=0$, and also $\xi_f(0) = 0, \xi'^2_f(0) = a^{2n}$.
Hence the value of $\Xi_f = \xi_f \delta'^2_f - \xi'^2_f$ at $t=0$ is $a^{2n}$. Hence for any $a \neq 0$, it is clear
that this polynomial cannot have common roots with $\delta_f$, hence $\rest(\delta_f,\Xi_f) \neq 0$.

For the case of $n \ge 4$ even, we will consider the polynomial $f+t = x^n + ax^{n-1} + bx + t$,
with $a,b \neq 0$. 
Let us write $n = 2m$. An easy computation yields $\delta_f(t) = a^mt^{m-1}+b^m$ and 
$\delta'_f(t) = (m-1)a^m t^{m-2}$, and a longer computation yields 
% \begin{equation*}
% \xi_f = \begin{cases}
%            a^{m+1}b^{m+1}t^{m-2} &m \equiv 0\quad(mod 4) \\
%            a^{m-1}b^{m-1}t^{m} &m \equiv 1\quad(mod 4) \\
%            a^{m+1}b^{m+1}t^{m-2} + \delta^2_f &m \equiv 2\quad(mod 4) \\
%            a^{m-1}b^{m-1}t^{m} + \delta^2_f &m \equiv 3\quad(mod 4) 
%           \end{cases}
% \end{equation*}
\begin{equation*}
\xi_f = \begin{cases}
           0 &m \equiv 0,1\pmod 4 \\
           \delta^2_f &m \equiv 2,3\pmod 4 
          \end{cases}
          +
          \begin{cases}
           a^{m+1}b^{m+1}t^{m-2} &m \equiv 0,2\pmod 4 \\
           a^{m-1}b^{m-1}t^{m} &m \equiv 1,3\pmod 4 
          \end{cases}
\end{equation*}
And hence
\begin{equation*}
\Xi_f = \xi_f \delta'^2_f - \xi'^2_f = \begin{cases}
           a^{3m+1}b^{m+1}t^{3m-6} &m \equiv 0\pmod 4 \\
           a^{3m+1}b^{m+1}t^{3m-6} + a^{n}t^{n-4}\delta^2_f &m \equiv 2\pmod 4 \\
           a^{n-2}b^{n-2}t^{n-2} &m \equiv 1\pmod 2
          \end{cases}
\end{equation*}
As $b \neq 0$, clearly $t = 0$ is not a root of $\delta_f$, but in all cases above,
it is either the sole root of $\Xi_f$ or of a combination of $\Xi_f$ and $\delta_f$. Either way,
it is clear that $\Xi_f$ and $\delta_f$ can have no common roots, and $\rest(\delta_f,\Xi_f) \neq 0$.
Thus we have shown inequality \eqref{A_n bound}.

\section{The case $n=2$}\label{sec:n2}
For $n=2$, the inequality \eqref{main thm eq} is not always valid -- 
sometimes there are correlations in the M\"obius function. 
The following proposition covers all cases where $n=2$:
\begin{prop}\label{n2 proposition}
Let $\alpha_1, \dots, \alpha_r\in \fq[x]$ be distinct linear polynomials $\alpha_i = a_ix+b_i$,
and let $\epsilon_1, \dots,\epsilon_r \in \{1,2\}$.
Set
% \begin{equation*}
% \begin{split}
\begin{align*}
&A = \{a_i : 1 \le i \le r\},&\\
&b_a = \sum_{i :\; a_i = a} \epsilon_ib_i,&
A_b = \{a \in A : b_a \neq 0 \},\\
&\gamma_a = \sum_{i :\; a_i = a} \epsilon_i \mod 2 ,&
A_\gamma = \{a \in A : \gamma_a \neq 0 \}
\end{align*}
% \end{split}
% \end{equation*}
One of the following relations then holds:
\begin{equation}
\begin{cases}
 |C(\alpha_1,\dots, \alpha_r;2)| < rq & A_\gamma \neq \emptyset \\
 |C(\alpha_1,\dots, \alpha_r;2)| < rq^\frac32 & A_\gamma = \emptyset, A_b \neq \emptyset \\
 \;\;C(\alpha_1,\dots, \alpha_r;2) \ge q^2 - rq & A_\gamma = \emptyset, A_b = \emptyset
\end{cases}
\end{equation}
\end{prop}
\begin{proof}
One may easily see that for a quadratic polynomial,
\begin{equation}
\Berl(x^2+ax+b) = \frac{b}{a^2} 
\end{equation}
In particular, $\mu(x^2+ax+b) = 0 \iff a = 0$, and otherwise
\begin{equation}
 \mu(x^2+ax+b) = \chi_2\Big(\frac{b}{a^2}\Big).
\end{equation}
Clearly for $f = x^2 + sx + t$, $\prod_{i} \mu(f+\alpha_i)^{\epsilon_i} = 0 \iff s \in A$,
so we may take our sum only over $s \not \in A$. 
There is no further contribution to the product from $\alpha_i$ where $\epsilon_i = 2$.
We compute:
\begin{equation}\label{n2 formula}
 \begin{split}
 C(\alpha_1,\dots, \alpha_r;2) &= \sum_{f\in M_2}\prod_{i} \mu(f+\alpha_i)^{\epsilon_i}
  = 
 \sum_{s \not \in A} \sum_{t \in \fq}\chi_2 \bigg(\sum_i \frac{\epsilon_i(b_i+t)}{(s+a_i)^2}\bigg) 
  \\ &=
 \sum_{s \not \in A} \sum_{t \in \fq}\chi_2 \bigg(\sum_{a \in A} \frac{b_a+\gamma_at}{s^2+a^2}\bigg) 
  \\ &=
 \sum_{s \not \in A} \sum_{t \in \fq}\chi_2 \bigg(\sum_{a \in A} \frac{b_a}{s^2+a^2}\bigg)\chi_2 \bigg(\sum_{a \in A} \frac{\gamma_a}{s^2+a^2}t\bigg) 
  \\ &=
 \sum_{s \not \in A} \chi_2 \bigg(\sum_{a \in A_b} \frac{b_a}{s^2+a^2}\bigg)\bigg(\sum_{t \in \fq}\chi_2 \Big(\big(\sum_{a \in A_\gamma} \frac{1}{s^2+a^2}\big)\;t\Big)\bigg)
 \end{split}
\end{equation}
Note that for any constant $c$,
\begin{equation*}
 \sum_{t \in \fq}\chi_2 \big(c t\big) = 
 \begin{cases}
 0 & c \neq 0 \\
 q & c = 0
 \end{cases}
\end{equation*}
We now have two cases. If $A_\gamma \neq \emptyset$, then
$\sum_{a \in A_\gamma} \frac{1}{s^2+a^2} = 0$ for at most $\#A_\gamma-1 < r$ values of $s$.
Hence in this case we have 
\begin{equation}
 |C(\alpha_1,\dots, \alpha_r;2)| \le (\#A_\gamma - 1)q < rq
\end{equation}
which is the first case of proposition \ref{n2 proposition}. If on the other hand $A_\gamma$ is empty,
then \eqref{n2 formula} becomes
\begin{equation}\label{n2 formula 2}
 C(\alpha_1,\dots, \alpha_r;2) = q \sum_{s \not \in A} \chi_2 \Big(\sum_{a \in A_b} \frac{b_a}{s^2+a^2}\Big)
\end{equation}
Once again, we have two cases. If $A_b = \emptyset$, then clearly
\begin{equation}
 C(\alpha_1,\dots, \alpha_r;2) = q(q-\#A) \ge q^2 - rq
\end{equation}
i.e., there is full correlation - every term in the sum is either 0 or 1. 
This is the third case of proposition \ref{n2 proposition}.
Finally, we are left with the case $A_b \neq \emptyset$. By the change of variables 
$y \leftarrow y + \sum_{a \in A_b} \frac{\sqrt{b_a}}{s+a}$, 
we see that the curve $y^2 + y = \sum_{a \in A_b} \frac{b_a}{s^2+a^2}$ is equivalent to
the curve $y^2 + y = \sum_{a \in A_b} \frac{\sqrt{b_a}}{s+a}$. The rational function 
$\sum_{a \in A_b} \frac{\sqrt{b_a}}{s+a}$ has exactly $\#A_b$ distinct simple poles, hence
by Theorem \ref{weil 2}, the genus of these curves is exactly $\#A_b-1$.
Note that $A_b \le \frac{r}{2}$: indeed, $A_\gamma = \emptyset$ implies
that each $a \in A_b$ is represented at least twice
in the sequence $\{a_i\}$. 
Applying Theorem \ref{weil 2} to equation \eqref{n2 formula 2} then yields
\begin{equation}
 |C(\alpha_1,\dots, \alpha_r;2)| \le (2(\#A_b-1)q^{\frac12}+1)q < rq^\frac32
\end{equation}
Completing the proof of proposition \ref{n2 proposition}.
\end{proof}

\section{Square Independence}

In \cite[Proposition 3.1]{LBS}, Bary-Soroker computes the following Galois group:
\begin{prop}[Bary-Soroker]\label{snr prop}
 Let q be an odd prime power, let n, r be positive integers, let $\mathbf{\alpha}
 = (\alpha_1,\dots,\alpha_r) \in \fq[x]^r$ be an $r$-tuple of distinct polynomials each
 of degree $< n$, let $\mathbf{U} = (U_0,\dots,U_{n-1})$ be an $n$-tuple of variables over $\fq$,
 and let $\mathcal{F} = x^n + U_{n-1}x^{n-1} + \dots + U_0 \in \fq[\mathbf{U},x]$. For each $i = 1,\dots,r$,
 let $\mathcal{F}_i = \mathcal{F}+\alpha_i$. Let ${\tilde \FF}_q$ be an algebraic closure of $\fq$, 
 let $E = {\tilde \FF}_q(\mathbf{U})$, let $F_i$ be the splitting
 field of $\mathcal{F}_i$ over $E$, and let $F$ be the splitting field of
 $\prod_{i=1}^r \mathcal{F}_i$ over $E$. Then $\gal(F/E) \cong S_n^r$. 
\end{prop}
We wish to extend this computation to even $q$ as well. The reliance on odd $q$ lies in the
following lemma (\cite[Lemma 3.3]{LBS}):
\begin{lem}\label{indep lem}
 For a separable polynomial $f \in E[x]$, denote by $\delt_x(f)$ the square class of its 
 discriminant $\disc_x(f)$ in the $\FF_2$-vector space $E^\times/(E^\times)^2$. The square classes
 $\delt_x(\mathcal{F}_1), \dots, \delt_x(\mathcal{F}_r)$ are linearly independent.
\end{lem}
The lemma is proven using some of the arguments from \cite{Carmon Rudnick}. It then sets the ground
for the application of a final lemma (\cite[Lemma 3.4]{LBS 2}):
\begin{lem}\label{disj lem}
 If the square classes $\delt_x(\mathcal{F}_1), \dots, \delt_x(\mathcal{F}_r)$ are linearly independent,
 then $F_1, \dots, F_r$ are linearly disjoint over $E$.
\end{lem}
Lemma \ref{disj lem}, together with $F$ being the compositum of $F_1,\dots,F_r$ and with
the classical fact that $\gal(F_i/E) \cong S_n$, easily yields Proposition \ref{snr prop}; see \cite[Section 3]{LBS}
for the full details. 

We note that Lemma \ref{disj lem} was proven in \cite{LBS 2} also in 
characteristic 2. In this case, $\delt_x(f)$ needs to be defined in terms of Berlekamp's discriminant, 
as the residue class $[\Berl(f)]$ in $E/\wp(E)$, where $\wp(y) = y^2 + y$. We shall prove Lemma \ref{indep lem}
in this context, for $n > 2$, in analogy to the proof in \cite{LBS}. This is the last required piece in the proof
of Proposition \ref{snr prop} for characteristic 2.
\begin{proof}
 Consider an arbitrary extension $\FF_{q^e}$, write $t = U_0$, and consider specializations 
 $(U_1,\dots,U_{n-1}) \mapsto a=(a_1,\dots,a_{n-1}) \in \FF_{q^e}^{n-1}$. The specialization of the polynomial 
 $\mathcal{F}$ is 
 $f = x^n + a_{n-1}x^{n-1} + \dots + a_1x + t \in \FF_{q^e}[x,t]$. $E/\wp(E)$ is specialized into
 $\FF_{q^e}(t)/\wp(\FF_{q^e}(t))$, and we shall work in this quotient.
 Examine the proof of Proposition \ref{main prop}. In its course we have in fact shown that if
 $[\Berl(f+\alpha_1)] \neq 0$ (i.e. $a \not \in A_n\cup C(1)$), and $\disc_x(f+\alpha_1)$ is coprime to $\disc_x(f+\alpha_j)$
 for each $j \neq 1$ (i.e. $a \not \in B_n\cup C_n$), then $[\Berl(f+\alpha_1)] + \dots + [\Berl(f+\alpha_r)] \neq 0$ (i.e.
 $a \in G_n$). We have also demonstrated that this must be the case for all but at most 
 $\frac 34 (r+3)n^2 q^{e(n-2)}$ specializations. The same arguments imply that if we further require that
 $[\Berl(f+\alpha_i)] \neq 0$ for each $i = 1, \dots, r$, and that $\disc_x(f+\alpha_i)$ is coprime to $\disc_x(f+\alpha_j)$
 for all $j \neq i$, then for any non-empty set of indices $I \subset \{1,\dots,r\}$, we have 
 $\sum_{i \in I} [\Berl(f+\alpha_i)] \neq 0$, which is equivalent to stating that $[\Berl(f+\alpha_1)], \dots, [\Berl(f+\alpha_r)]$
 are linearly independent in $\FF_{q^e}(t)/\wp(\FF_{q^e}(t))$. Furthermore, this must be the case for all but at most 
 $\frac 34 r(r+3)n^2 q^{e(n-2)}$ of the $q^{e(n-1)}$ specializations $a \in \FF_{q^e}^{n-1}$. In particular, for large enough exponent $e$,
 it is clear that there must exist at least one such specialization. But since a linear dependence in $E/\wp(E)$ would
 survive specialization, it follows that the set of $\delt_x(\mathcal{F}_i) = [\Berl(\mathcal{F}+\alpha_i)]$ must indeed be
 linearly independent in $E/\wp(E)$.
\end{proof}
\begin{remark}
 Lemma \ref{indep lem} is not true in general for $n = 2$, where linear dependence of the classes is possible, as demonstrated
 in the third case of Proposition \ref{n2 proposition}. The result is still valid under more specific conditions, that rule
 out the possibility of dependence occurring in any subset of $\alpha_1,\dots,\alpha_r$. Specifically, we require
 that there is no non-empty subset $S \subset \{1,\dots,r\}$ such that $\{\alpha_i : i \in S\}$ all share the same linear
 term and $\sum_{i \in S} (x^2 + \alpha_i) = 0$.
\end{remark}
7
\section{Computations}\label{sec:computation}
In this section we calculate $\delta(f),\ \xi(f)$ for
the polynomials $f = x^n + a x^2 + t$ with $n$ odd, and $f = x^n + ax^{n-1}+bx+t$ with $n$ even.
To do so, we will compute $\disc(f)$ and $\pisc(f)$. Writing $f(x) = g(x^2)+xh(x^2)$,
we see from \eqref{pisc formula 2}, that (working modulo 2), 
\begin{equation}
\delta^2(f) = \disc(f) = \pisc(f) = \rest(h(x),g(x))^2 
\end{equation}
Hence $\delta(f) = \rest(h(x),g(x))$. The formula for resultants of binomials is well known (see \cite[Lemma 3]{Swan})
and we obtain
\begin{eqnarray*}
&\delta(x^n + a x^2 + t)  =  \rest(x^{\frac{n-1}2},ax+t) = t^{\frac{n-1}2} \\
%\end{equation*}
%\begin{equation*}
&\delta(x^n + ax^{n-1}+bx+t)  =  \rest(ax^{\frac{n}2-1}+b,x^{\frac{n}2}+t)
= a^{\frac{n}2}t^{\frac{n}2-1}+b^{\frac{n}2}
\end{eqnarray*}
Note that the above computations for the resultants are valid also in characteristic 0,
so we have also computed $\pisc(f)$:
\begin{align*}
&\pisc(x^n + a x^2 + t) =  \rest(x^{\frac{n-1}2},ax+t)^2 = t^{n-1} \\
&\begin{aligned}\pisc(x^n + ax^{n-1}+bx+t) &= \rest(ax^{\frac{n}2-1}+b,x^{\frac{n}2}+t)^2 \\
&= a^nt^{n-2}+b^n+2a^{\frac{n}2}b^{\frac{n}2}t^{\frac{n}2-1}
\end{aligned}
\end{align*}
We now need to compute $\disc(f)$ in characteristic 0. 

For odd $n$, we apply the general formula for trinomial discriminants (\cite[Theorem 2]{Swan})
and obtain
% \begin{equation*}
%  \begin{split}
%   &\disc(x^n+ax^2+t) = (-1)^{\frac{(n-1)n}2}\rest(x^n+ax^2+t,nx^{n-1}+2ax) \\
%   &= (-1)^{\frac{n-1}2}\rest(0x^n+(1-\tfrac2n)ax^2+t,nx^{n-1}+2ax) \\
%   &= (-1)^{\frac{n-1}2}n^{n-2}\rest(\tfrac{n-2}nax^2+t,nx^{n-1}+2ax) \\
%   &= (-1)^{\frac{n-1}2}n^{n-2}t\rest(\tfrac{n-2}nax^2+t,nx^{n-2}+2a) \\
%   &= (-1)^{\frac{n-1}2}n^{n-2}t\big((\tfrac{n-2}na)^{n-2}(2a)^2 + n^2t^{n-2}\big) \\
%   &= (-1)^{\frac{n-1}2}\big(n^nt^{n-1}+4(n-2)^{n-2}a^nt\big) 
%  \end{split}
% \end{equation*}
\begin{equation*}
 \begin{split}
  &\disc(x^n+ax^2+t) = (-1)^{\frac{n-1}2}\big(n^nt^{n-1}+4(n-2)^{n-2}a^nt\big) 
 \end{split}
\end{equation*}
It is now immediate to compute 
\begin{equation*}
\xi(x^n+ax^2+t) = \frac{\pisc - \disc}{4} \mod 2 = \frac{(-1)^\frac{n-1}{2}n^n-1}{4}t^{n-1}+a^nt.
\end{equation*}
The computation for even $n$ is somewhat longer and more complicated:
\allowdisplaybreaks
%\begin{equation*}
 \begin{align*}
  &\disc(x^n+ax^{n-1}+bx+t) \\
  &= (-1)^{\frac{(n-1)n}2}\rest(x^n+ax^{n-1}+bx+t,nx^{n-1}+(n-1)ax^{n-2}+b) \\
  &= (-1)^{\frac{n}2}\tfrac1t\rest(x^n+ax^{n-1}+bx+t,nx^{n}+(n-1)ax^{n-1}+bx) \\
  &= (-1)^{\frac{n}2}\tfrac1t\rest(x^n+ax^{n-1}+bx+t,ax^{n-1}+(n-1)bx+nt) \\
  &= (-1)^{\frac{n}2}\tfrac1t\rest(x^n+(2-n)bx+(1-n)t,ax^{n-1}+(n-1)bx+nt) \\ 
  &= (-1)^{\frac{n}2+1}\tfrac1ta^n\rest(-x^n+(n-2)bx+(n-1)t,x^{n-1}+\tfrac{(n-1)b}ax+\tfrac{nt}a) \\
  &
  \begin{aligned}
   = (-1)^{\frac{n}2+1}\tfrac1ta^n\rest\big(
   & \tfrac{(n-1)b}ax^2+((n-2)b+\tfrac{nt}a)x+(n-1)t, \\
   & x^{n-1}+\tfrac{(n-1)b}ax+\tfrac{nt}a\big)     
  \end{aligned} \\
  &
  \begin{aligned}
   =(-1)^{\frac{n}2+1}\tfrac1{(n-1)t^2}a^n\rest\big(
   &\tfrac{(n-1)b}ax^2+((n-2)b+\tfrac{nt}a)x+(n-1)t, \\
   & x^n-(n-2)bx-(n-1)t\big)     
  \end{aligned}\\
  &
  \begin{aligned}
      =(-1)^{\frac{n}2+1}\tfrac1{(n-1)t^2}\rest(
      & (n-1)bx^2+((n-2)ab+nt)x+(n-1)ta,\\
      & x^n-(n-2)bx-(n-1)t)
   \end{aligned} \\
  &= (-1)^{\frac{n}2+1}\tfrac1{(n-1)t^2}\rest(\alpha x^2+\beta x+\gamma,x^n-ux-v)
 \end{align*}
%\end{equation*}
% \begin{align*}
%   &\disc(x^n+ax^{n-1}+bx+t&) \\
%   &= (-1)^{\frac{(n-1)n}2}\rest(x^n+ax^{n-1}+bx+t,nx^{n-1}+(n-1)ax^{n-2}+b)& \\
%   &= (-1)^{\frac{n}2}\tfrac1t\rest(x^n+ax^{n-1}+bx+t,nx^{n}+(n-1)ax^{n-1}+bx)& \\
%   &= (-1)^{\frac{n}2}\tfrac1t\rest(x^n+ax^{n-1}+bx+t,ax^{n-1}+(n-1)bx+nt)& \\
%   &= (-1)^{\frac{n}2}\tfrac1t\rest(x^n+(2-n)bx+(1-n)t,ax^{n-1}+(n-1)bx+nt&) \\
%   &= (-1)^{\frac{n}2+1}\tfrac1ta^n\rest(-x^n+(n-2)bx+(n-1)t,x^{n-1}+\tfrac{(n-1)b}ax+\tfrac{nt}a) \\ 
%   &= (-1)^{\frac{n}2+1}\tfrac1ta^n\rest(\tfrac{(n-1)b}ax^2+((n-2)b+\tfrac{nt}a)x+(n-1)t,
%   x^{n-1}+\tfrac{(n-1)b}ax+\tfrac{nt}a) \\
%   &= (-1)^{\frac{n}2+1}\tfrac1{(n-1)t^2}a^n\rest(&\tfrac{(n-1)b}ax^2+((n-2)b+\tfrac{nt}a)x+(n-1)t,
%   \\ && x^n-(n-2)bx-(n-1)t) \\ 
%   &= (-1)^{\frac{n}2+1}\tfrac1{(n-1)t^2}\rest((n-1)bx^2+((n-2)ab+nt)x+(n-1)ta,
%   x^n-(n-2)bx-(n-1)t) \\
%   &= (-1)^{\frac{n}2+1}\tfrac1{(n-1)t^2}\rest(\alpha x^2+\beta x+\gamma,x^n-ux-v)
% \end{align*}
where $\alpha = (n-1)b,\ \beta = ((n-2)ab+nt),\ \gamma=(n-1)ta,\ u=(n-2)b$, and $v=(n-1)t$ are defined by the last equivalence. 
Note that $\alpha,\gamma,v$ are odd and $\beta,u$ are even. 
The number of terms in the last resultant is unbounded as $n$ increases.
However, we are only interested in $\disc(f) \mod 8$, 
and as $\beta$ is even, only finitely many of the terms will be non-zero modulo 8.
Therefore we will continue our computation modulo 8.

Let $r_{1,2} = \frac{-\beta \pm \sqrt{\beta^2-4\alpha\gamma}}{2\alpha} $ be the roots of 
$\alpha x^2+\beta x+\gamma$. Let $R_k = \alpha^k(r_1^k+r_2^k)$.
It is easy to see that 
\begin{equation*}
 R_0 = 2,\quad
 R_1 = -\beta,\quad
 \forall k \ge 2, R_k = -\beta R_{k-1} - \alpha\gamma R_{k-2}
\end{equation*}
Thus for all $k$, $R_k$ is a polynomial in $\beta,\alpha\gamma$. Using the recursion formula,
one may show by induction that
\begin{equation}
 R_k = \sum_{0 \le l \le k/2} (-1)^{k-l}\big(\tbinom{k-l}{l}+\tbinom{k-l-1}{l-1}\big)(\alpha\gamma)^{l}\beta^{k-2l}
\end{equation}

As we are only interested in $R_k \mod 8$,  we may discard all monomials
containing powers of $\beta$ greater than 3, and obtain:
\begin{equation}
 R_k \mod \beta^3  = \begin{cases}
        (-1)^m\big(2(\alpha\gamma)^m - m^2(\alpha\gamma)^{m-1}\beta^2\big) & k = 2m \text{ even}\\
        (-1)^mk(\alpha\gamma)^{m-1}\beta & k = 2m-1 \text{ odd} 
       \end{cases}
\end{equation}

We now compute:
\begin{equation*}
 \begin{split}
 & \rest(\alpha x^2+\beta x+\gamma,x^n-ux-v) 
 = \alpha^n(r_1^n-ur_1-v)(r_2^n-ur_2-v) \\
 & = \alpha^n \big((r_1r_2)^n + u^2r_1r_2+ v^2 -ur_1r_2(r_1^{n-1}+r_2^{n-1}) - v(r_1^n+r_2^n) + uv(r_1+r_2)\big) \\
 & = \gamma^n + u^2\alpha^{n-1}\gamma+v^2\alpha^n - u\gamma R_{n-1} - vR_n - uv\alpha^{n-1}\beta
 \end{split}
\end{equation*}

We can now replace $\alpha,\beta,\gamma,u,v,R_{n-1},R_n$ by their actual values.
We first expand the part of the expression not involving $R_{n-1}, R_n$:
\begin{equation*}
 \begin{split}
  & \gamma^n + u^2\alpha^{n-1}\gamma+v^2\alpha^n - uv\alpha^{n-1}\beta = 
  \gamma^n + \alpha^{n-1}(u^2\gamma+v^2\alpha-uv\beta) = \\
  & (n-1)^{n}\big(a^nt^n + b^{n-1}\big((n-2)^2ab^2t +(n-1)^2bt^2 -(n-2)bt((n-2)ab+nt)\big)\big) = \\
  % & (n-1)^{n}\Big(a^nt^n + b^{n-1}\big(((n-2)^2 - (n-2)^2) ab^2t +((n-1)^2-(n-2)n)bt^2\big)\Big) = \\
  & (n-1)^n(a^nt^n + b^nt^2) 
 \end{split}
\end{equation*}

Next, set $n = 2m$, and expand
\begin{equation*}
 \begin{split}
  & (-1)^m(u\gamma R_{n-1} + vR_n) = \\
  & (n-1)u\alpha^{m-1}\gamma^m\beta + 2v\alpha^m\gamma^m - m^2v\alpha^{m-1}\gamma^{m-1}\beta^2 \equiv \\
  & (\alpha\gamma)^{m-1}(u\gamma\beta + 2v\alpha\gamma -m^2v\beta^2) = \\
  & (n-1)^{n-1}(abt)^{m-1}((n-2)(n-1)abt\beta + 2(n-1)^2abt^2 -m^2t\beta^2) \equiv \\
  & 2(n-1)a^mb^mt^{m+1} + (n-2)^2 a^{m+1}b^{m+1}t^m - m^2n^2a^{m-1}b^{m-1}t^{m+2} \pmod 8
 \end{split}
\end{equation*}
where the cancellations in the last congruence are due to identities such as $(n-1)^2 \equiv 1 \pmod 8$, 
$n(n-2) \equiv 0 \pmod 8$ and $ax \equiv a \pmod 8$ when $a$ is divisible by 4 and $x$ is odd.
Note that of the last two terms, exactly one has coefficient $\equiv 0 \pmod 8$, and the other coefficient
$\equiv 4 \pmod 8$, determined by the parity of $m$.

Combining all expansions, we obtain
\begin{equation*}
 \begin{split}
  &\disc(x^n+ax^{n-1}+bx+t) \\
  &= (-1)^{\frac{n}2+1}\tfrac1{(n-1)t^2}\rest(\alpha x^2+\beta x+\gamma,x^n-ux-v) \\
  &\equiv (-1)^{m+1}(n-1)^{n-1}\big(a^nt^{n-2}+b^n) + 2a^mb^mt^{m-1} \\
  &\quad+4\begin{cases}
    a^{m+1}b^{m+1}t^{m-2} & m \equiv 0 \pmod 2\\
    a^{m-1}b^{m-1}t^{m} & m \equiv 1 \pmod 2
   \end{cases}
 \end{split}
\end{equation*}

Noting also that $a^nt^{n-2}+b^n = \delta_f^2$ and 
\begin{equation}
 (-1)^{m+1}(n-1)^{n-1} \mod 8 = 
 \begin{cases}
  1 & m \equiv 0,1 \pmod 4 \\
  5 & m \equiv 2,3 \pmod 4
 \end{cases}
\end{equation}

we can now calculate 
\begin{equation*}
\begin{split}
 & \xi(x^n+ax^{n-1}+bx+t) = \frac{\pisc - \disc}{4} \mod 2 = \\
 & \begin{cases}
           0 &m \equiv 0,1\pmod 4 \\
           \delta^2_f &m \equiv 2,3\pmod 4 
          \end{cases}
          +
          \begin{cases}
           a^{m+1}b^{m+1}t^{m-2} &m \equiv 0,2\pmod 4 \\
           a^{m-1}b^{m-1}t^{m} &m \equiv 1,3\pmod 4 
          \end{cases}
\end{split}
\end{equation*}

as claimed.\qed

\section*{Acknowledgements}
The author wishes to thank Lior Bary-Soroker for introducing him to Berlekamp's discriminant 
and for helpful discussions, 
and to thank Ze\`ev Rudnick for his guidance and many helpful suggestions.

\end{document}